\documentclass[leqno, 12pt]{article}
\usepackage{url}
\usepackage{pslatex}
\usepackage{amsmath,amsthm,amssymb,amsfonts, stmaryrd}
\usepackage{fancybox}
\usepackage{color, graphicx}
\usepackage[T1]{fontenc}
\usepackage{soul}

\usepackage{hyperref}

\usepackage[usenames,dvipsnames,svgnames,table]{xcolor}
\newtheorem{theorem}{Theorem}[section]

\newtheorem{corollary}[theorem]{Corollary}

\newtheorem{definition}[theorem]{Definition}
\newtheorem{example}[theorem]{Example}

\newtheorem{notation}[theorem]{Notation}

\newtheorem{proposition}[theorem]{Proposition}
\newtheorem{question}[theorem]{Question}

\newcommand{\Z}{\mathbb Z}

\newcommand{\N}{\mathbb N}

\DeclareMathOperator{\buch}{\mathcal{B}}
\DeclareMathOperator{\buchset}{Buch}

\newcommand{\vs}[1]{\langle #1 \rangle}

\usepackage{tikz}
\usetikzlibrary{calc}

\begin{document}

\title{The Buchweitz set of a numerical semigroup}

\author{S. Eliahou, J.I. Garc\'ia-Garc\'ia, \\ D. Mar\'in-Arag\'on, A. Vigneron-Tenorio}
\date{}

\maketitle

\begin{abstract} Let $A \subset \Z$ be a finite subset. We denote by $\buch(A)$ the set of all integers $n \ge 2$ such that $|nA| > (2n-1)(|A|-1)$, where $nA=A+\cdots+A$ denotes the $n$-fold sumset of $A$. The motivation to consider $\buch(A)$ stems from Buchweitz's discovery in 1980 that if a numerical semigroup $S \subseteq \N$ is a Weierstrass semigroup, then $\buch(\N \setminus S) = \emptyset$. By constructing instances where this condition fails, Buchweitz disproved a longstanding conjecture by Hurwitz (1893). In this paper, we prove that for any numerical semigroup $S \subset \N$ of genus $g \ge 2$, the set $\buch(\N \setminus S) $ is finite, of unbounded cardinality as $S$ varies.
\end{abstract}

\begin{quote} \textit{Keywords and phrases.} Weierstrass numerical semigroup; gapset; additive combinatorics; sumset growth; Freiman's $3k-3$ theorem.

{\small \emph{MSC-class:} {14H55, 11P70, 20M14}}
\end{quote}

\begin{center}\emph{
    To the memory of our friend and colleague\\
    Fernando Torres (1961-2020)}
\end{center}

\section{Introduction}\label{section intro}
Denote $\N=\{0,1,2,3,\dots\}$ and $\N_+=\N\setminus \{0\}=\{1,2,3,\dots\}$. For $a,b \in \Z$, let $[a,b[=\{z \in \Z \mid a \le z < b\}$ and $[a,\infty[=\{z \in \Z \mid a \le z\}$ denote the integer intervals they span. A \emph{numerical semigroup} is a subset $S \subseteq \N$ containing $0$, stable under addition and with finite complement in $\N$. Equivalently, it is a subset of $\N$ of the form $S = \vs{a_1,\dots,a_n}=\N a_1 + \dots + \N a_n$ where $\gcd(a_1,\dots,a_n)=1$. The set $\{a_1,\dots,a_n\}$ is then called a \emph{system of generators} of $S$, and the smallest such $n$ is called the \emph{embedding dimension} of $S$.

For a numerical semigroup $S$, its corresponding \emph{gapset} is the complement $G=\N \setminus S$, its \emph{genus} is $g=|G|$, its \emph{multiplicity} is $m = \min S^*$ where $S^*= S \setminus \{0\}$, its \emph{Frobenius number} is $f = \max(\Z\setminus S)$ and its \emph{conductor} is $c=f+1$. Thus $[c,\infty[ \, \subseteq S$ and $c$ is minimal for this property. Finally, the \emph{depth} of $S$ is $q=\lceil c/m \rceil$.

Given a finite subset  $A \subset \N$, we denote by $nA=A+\dots+A$ the \emph{$n$-fold sumset of $A$.} See Section~\ref{sec sumset} for more details.

\begin{definition} Let $A \subset \Z$ be a finite subset. We associate to $A$ the function $\beta=\beta_A \colon \N_+ \to \Z$ defined for all $n \ge 1$ by
$$
\beta_A(n) = |nA|-(2n-1)(|A|-1).
$$
\end{definition}

\begin{notation}
We denote by $\buch(A)$ the \emph{positive support} in $2+\N$ of the function $\beta_A$, i.e.
$$
\buch(A) = \{n \ge 2 \mid \beta_A(n) \ge 1\}.
$$
\end{notation}

For instance, $2 \in \buch(A)$ if and only if $|2A| \ge 3|A|-2$. Interestingly, the failure of this condition, namely the inequality $|2A| \le 3|A|-3$, is the key hypothesis of the famous Freiman's $3k-3$ Theorem in additive combinatorics~\cite{F}.

\smallskip

\begin{example}
If $|A|=0$ or $1$, then $\buch(A)$ is infinite. Indeed, if $A=\emptyset$, then $|nA|=0$ and so $\beta_\emptyset(n)=2n-1$ for all $n \ge 1$. Thus $\buch(\emptyset)=2+\N$ in that case. Similarly, if $|A|=1$, then $\beta_A(n)=1$ for all $n \ge 1$. So here again $\buch(A)=2+\N$.
\end{example}

In sharp contrast, Theorem~\ref{finite for NS} below states that \emph{if $S \subset \N$ is a numerical semigroup of genus $g \ge 2$, then $\buch(\N \setminus S)$ is finite}.
\begin{example} Let $S=\vs{3,7}$. Then $\N \setminus S=\{1,2,4,5,8,11\}$ and $\beta_{\N \setminus S}(n)=0$ for all $n \ge 2$ as easily seen. In particular, $\buch(\N \setminus S)=\emptyset$.
\end{example}
More generally, it was shown in \cite{K,O} that $\beta_{\N \setminus S}(n)=0$ for all \emph{symmetric} numerical semigroups $S$ of multiplicity $m \ge 3$ and all $n \ge 2$. We shall not use this result below, but instead give a short self-contained proof of an immediate consequence, namely that $\buch(\N \setminus S)$ is empty in that case.

\smallskip
In fact, $\buch(\N \setminus S)$ is empty in most cases. Indeed, Buchweitz discovered in 1980 that the condition $\buch(\N \setminus S) = \emptyset$ is necessary for $S$ to be a Weierstrass semigroup. By constructing instances where this condition fails,  Buchweitz~\cite{B} was able to negate the longstanding conjecture by Hurwitz~\cite{H} according to which all numerical semigroups of genus $g \ge 2$ are Weierstrass semigroups. His first counterexample was $S =\vs{13,14,15,16,17,18,20,22,23}$, with corresponding gapset $$G = \N \setminus S= [1,12] \cup \{19, 21, 24, 25\}$$ of cardinality $16$. Then $2G=[2,50] \setminus \{39, 41, 47\}$, so that $|2G|=46$ and $\beta_G(2)=46-3 \cdot 15=1$, implying $2 \in \buch(G)$ and thus impeding $S$ to be a Weierstrass semigroup. For more information on Buchweitz's condition and Weierstrass semigroups, see e.g.~\cite{EH, KY}.

Here are the contents of this paper. In Section~\ref{sec sumset}, we recall a result of Nathanson in additive combinatorics and we use it to study the asymptotic behavior of the function $\beta_A(n)$. In Section~\ref{sec numsemi}, we introduce the Buchweitz set of a numerical semigroup and we prove our main results. Section~\ref{sec conclusion} concludes the paper with open questions on the possible shapes of the sets $\buch(A)$.

\section{Sumset growth}\label{sec sumset}

Given finite subsets $A,B$ of a commutative monoid $(M,+)$, we denote as usual
$$
A+B=\{a+b \mid a \in A, b \in B\},
$$
the \emph{sumset} of $A,B$, and $2A=A+A$. More generally, if $n \ge 2$, we denote $nA=A+(n-1)A$, where $1A=A$. The set $nA$ is called the $n$-fold sumset of $A$.

\smallskip
A classical question in additive combinatorics is, how does $|nA|$ grow with $n$? Here we only consider the case $M=\Z$. We shall need the following result of Nathanson~\cite[Theorem 1.1]{N}.

\begin{theorem}\label{nathanson} Let $A_0 \subset \N$ be a finite subset of cardinality $k \ge 2$, containing $0$ and such that $\gcd(A_0)=1$. Let $a_0=\max (A_0)$. Then there exist integers $c,d$ and subsets $C \subseteq [0,c-2],\, D \subseteq [0,d-2]$ such that
$$
nA_0 = C \sqcup [c,a_0n-d] \sqcup (a_0n-D)
$$
for all $n \ge \max\{(|A_0|-2)(a_0-1)a_0,1\}$.
\end{theorem}

As pointed out in~\cite{N}, the hypotheses $0 \in A_0$ and $\gcd(A_0)=1$ are not really restrictive. Indeed, for any finite set $A \subset \Z$ with $|A| \ge 2$, the simple transformation $A \mapsto A_0=(A-\alpha)/d$, where $\alpha=\min(A)$ and $d=\gcd(A-\alpha)$, yields a set $A_0$ satisfying these hypotheses and such that $|nA_0|=|nA|$ for all $n$. In view of our applications to gapsets, we shall need the following version.
\begin{corollary}\label{growth} Let $A \subset \N_+$ be a finite subset containing $\{1,2\}$. Let $a=\max(A)$. Then there is an integer $b \le 1$ such that
$$
|nA| = (a-1)n+b
$$
for all $n \ge (|A|-2)(a-2)(a-1)$.
\end{corollary}
\begin{proof} Set $A_0=A-1$ and $a_0=a-1$. Then $A_0$ contains $\{0,1\}$, hence it satisfies the hypotheses of Theorem~\ref{nathanson}. Using the same notation, its conclusion implies
\begin{equation}\label{nA0}
|nA_0| = a_0n+b
\end{equation}
for all  $n \ge \max\{(|A_0|-2)(a_0-1)a_0,1\}$, where
$$b=|C|+|D|-c-d+1.$$ Note that $b \le 1$ since $|C| \le \max(0,c-1)$, $|D| \le \max(0,d-1)$. The desired statement follows from \eqref{nA0} since $|nA|=|nA_0|$ for all $n \ge 0$.
\end{proof}

\subsection{Asymptotic behavior of $\beta_A(n)$}
We now study the evolution of $\beta_A(n)$ as $n$ grows.
\begin{theorem}\label{long term} Let $A \subset \N_+$ be a finite set containing $\{1,2\}$. Let $f=\max(A)$ and $g=|A|$. Then
$$
\lim_{n \to \infty} \beta_A(n) = \left\{
\begin{array}{rcl}
-\infty & \textrm{if} & f \le 2g-2,\\
+\infty & \textrm{if} & f \ge 2g.
\end{array}
\right.
$$
Finally if $f=2g-1$, then $\beta_A(n)$ is constant and nonpositive for $n$ large enough.
\end{theorem}
\begin{proof} By Corollary~\ref{growth}, we have $|nA|=(f-1)n+b$ for some integer $b \le 1$ and for $n$ large enough. Hence
\begin{eqnarray*}
\beta_A(n) & = & (f-1)n+b-(2n+1)(g-1) \\
& = & (f-2g+1)n + b+1-g
\end{eqnarray*}
for $n$ large enough. The claims for $f \le 2g-2$ and $f \ge 2g$ follow. If $f=2g-1$, then $\beta_A(n)=b+1-g \le 0$ for $n$ large enough, since $b \le 1$ and $g \ge 2$.
\end{proof}

\begin{corollary}\label{cor finite} Let $A \subseteq \N_+$ be a finite set containing $\{1,2\}$. Let $f=\max(A)$ and $g=|A|$. Then $\buch(A)$ is finite if and only if $f \le 2g-1$.
\end{corollary}
\begin{proof} If $f \ge 2g$, then $\lim_{n \to \infty} \beta_{A}(n)=\infty$ by the theorem, whence $\beta_A(n) \ge 1$ for all large enough $n$. Thus $\buch(A)$ is infinite in this case. If $f \le 2g-1$, the theorem implies $\beta_A(n) \le 0$ for $n$ large enough, whence $\buch(A)$ is finite in that case.
\end{proof}

\section{Application to numerical semigroups}\label{sec numsemi}

\begin{definition} Let $S \subseteq \N$ be a numerical semigroup. We define the \emph{Buchweitz set of $S$} as
$\buchset(S) = \buch(\N \setminus S)$. Explicitly, setting $G=\N\setminus S$, we have
\begin{eqnarray*}
\buchset(S) & = & \{n \ge 2 \mid |nG| > (2n-1)(|G|-1)\} \\
& = & \{n \ge 2 \mid \beta_G(n) \ge 1\}.
\end{eqnarray*}
\end{definition}

\smallskip
In this section, we first prove that $\buchset(S)$ is \emph{finite} for all numerical semigroups $S$ of genus $g \ge 2$. We then show, by explicit construction, that the cardinality of $\buchset(S)$ may be arbitrarily large.

\subsection{Finiteness of $\buchset(S)$}

We start with a well known inequality linking the Frobenius number and the genus of a numerical semigroup.
\begin{proposition}\label{prop f le 2g-1}
Let $S \subset \N$ be a numerical semigroup with Frobenius number $f$ and genus $g \ge 1$. Then $f \le 2g-1$.
\end{proposition}
\begin{proof} Let $x \in S \cap [0,f]$. Then $f-x \notin S$ since $S$ is stable under addition and $x+(f-x)=f \notin S$. Hence, the map $x \mapsto f-x$ induces an injection
$$S  \cap [0,f] \hookrightarrow \N \setminus S.$$
Since $|S \cap [0,f]| = (f+1)-g$, it follows that $f \le 2g-1$, as claimed.
\end{proof}

Recall that $S$ is said to be \emph{symmetric} if $|S \cap [0,f]|=|\N\setminus S|$, i.e. if $f=2g-1$. A classical result of Sylvester states that any numerical semigroup of the form $S = \vs{a,b}$ with $\gcd(a,b)=1$ is symmetric.

\begin{theorem}\label{finite for NS} Let $S \subseteq \N$ be a numerical semigroup of genus $g\ge 2$. Then $\buchset(S)$ is finite.
\end{theorem}

\begin{proof} Let $G = \N \setminus S$. Then $\buchset(S) = \buch(G)$ by definition.  We have $g=|G| \ge 2$. Let $f=\max(G)$ be the Frobenius number of $S$. Let $m = \min(S \setminus \{0\})$ be the multiplicity of $S$. Then $m \ge 2$ since $g \ge 2$, and $[1,m-1] \subseteq G$.

Assume first $m \ge 3$. Then $\{1,2\} \subseteq G$. Hence Corollary~\ref{cor finite} applies, and since $f \le 2g-1$ by Proposition~\ref{prop f le 2g-1}, it yields that $\buch(G)$ is finite, as desired.

Assume now $m=2$. Then $S=\vs{2,b}$ with $b$ odd and $b \ge 5$ since $|G| \ge 2$. At this point, we might conclude the proof right away using what is known in the symmetric case \cite{K,O}. However, for the convenience of the reader, let us give a short self-contained argument. We have $G=\{1,3,\dots,b-2\}$, i.e. all odd numbers from $1$ to $b-2$. Hence $G-1=\{0,2,\dots,b-3\}$ and $\gcd(G-1)=2$. Set $A=(G-1)/2=[0,k]$, where $k=(b-3)/2$. For all $n \ge 1$, we have
$$|nG|=|nA|=|nk+1|=n(|G|-1)+1.$$
Therefore $\beta_G(n)=-(n-1)|G|+n$, whence $\beta_G(n) \le 0$ for all $n \ge 2$. It follows that $\buch(G) = \emptyset$ and we are done.
\end{proof}

\subsection{Unboundedness of $|\buchset(S)|$}

We show here, by explicit construction, that $|\buchset(S)|$ may be arbitrarily large.

\begin{proposition}\label{prop [2,k+2]}
For any integer $b\geq 3$, there exists a numerical semigroup $S$ such that $\buchset(S)=[2,b]$.
\end{proposition}

\begin{proof}
Let $k=b-2$, and let $S$ be the numerical semigroup of multiplicity $m=6k+15$ and depth $q=2$ whose  corresponding gapset $G=\N \setminus S$ is given by
\begin{equation}\label{eq gapset G}
G = [1,m-1] \sqcup \{2m-7, 2m-5, 2m-2, 2m-1\}.
\end{equation}
We claim that $\buchset(S)=[2,k+2]$. Indeed, we will show a more precise statement, namely
$$
\beta_G(n)=\left\{
\begin{array}{ccl}
1 & \textrm{if} & n=2,\\
2 & \textrm{if} & 3\leq n \leq k+2,\\
-6(n-k-3) & \textrm{if} & n\geq  k+3.
\end{array}
\right.
$$
Let $A=(2m-1)-G$. Then $\beta_G(n)=\beta_A(n)$ since $|nG|=|nA|$ for all $n \ge 1$. We have
$$
A=[0,1] \sqcup \{4,6\} \sqcup [m,2(m-1)].
$$
Let us compute $2A$ and $3A$. We obtain
\begin{eqnarray*}
2A & = & [0,2] \sqcup [4,8] \sqcup \{10,12\} \sqcup [m,4(m-1)], \\
3A & = & [0,14] \cup \{16,18\} \cup [m,6(m-1)].
\end{eqnarray*}
In general, we have
\begin{equation}\label{nA}
nA = \left([0,6n-4] \sqcup \{6n-2,6n\}\right) \cup [m,2n(m-1)]
\end{equation}
for all $n \ge 3$, as easily verified by induction on $n$.

Let us determine $|nA|$ for all $n \ge 1$.  Note first that \emph{the union in \eqref{nA} is disjoint if and only if $6n+1 \le m$ }. Moreover, as $m=6k+15$, we have
\begin{equation*}
    6n+1\leq m \Longleftrightarrow n\leq k+2.
\end{equation*}
In contrast, if $n\geq k+3$, i.e. if $6n-3\geq m$, then the union in (\ref{nA}) collapses to a single interval and we get
\begin{equation*}
    nA=[0,2n(m-1)].
\end{equation*}
Summarizing, we have
$$
|nA| = \left\{
\begin{array}{lcl}
    m+3 & \textrm{if} &  n=1,\\
    3m+7 & \textrm{if} &   n=2,\\
    (2n-1)(m-1)+6n-1 & \textrm{if} &  3\leq n\leq k+2,\\
    2n(m-1)+1 & \textrm{if} & n\geq k+4.\end{array}
\right.
$$
The stated formula for $\beta_G(n)=\beta_A(n)=|nA|-(2n-1)(|A|-1)$ follows. Hence $\buchset(S)=[2,k+2]$, as claimed.
\end{proof}
This family of numerical semigroups was inspired by the \emph{PF-semigroups} introduced in \cite{PFsemigroup}.

\subsection{More intervals}

What are the possible shapes of $\buchset(S)$ when $S$ varies? We do not know in general. By Proposition~\ref{prop [2,k+2]}, any finite integer interval $I$ with $|I| \ge 2$ and $\min(I) = 2$ may be realized as $I=\buchset(S)$ for some numerical semigroup $S$. Here we present families of numerical semigroups $S$ realizing as $\buchset(S)$ all finite integer intervals $I$ with $|I| \ge 2$ and $\min(I) \in \{3,4,5,6\}$.

\begin{proposition} Let $k \ge 1$. Let $S$ be the numerical semigroup of multiplicity $m=6k+19$ and depth $q=2$ whose corresponding gapset $G=\N \setminus S$ is given by
\begin{equation}\label{eq gapset G 3}
G = [1,m-1] \sqcup \{2m-7, 2m-6, 2m-2, 2m-1\}.
\end{equation}
Then $\buchset(S)=[3,k+3]$.
\end{proposition}
\begin{proof} Let again $A=(2m-1)-G=[0,1]\sqcup\{5,6\}\sqcup[m,2(m-1)]$.
We then have
\begin{eqnarray*}
2A & = & [0,2] \sqcup [5,7] \sqcup [10,12] \sqcup [m,4(m-1)], \\
3A & = & [0,3] \sqcup [5,8] \sqcup [10,13] \sqcup [15,18] \sqcup [m,6(m-1)], \\
4A & = & [0,24] \cup [m,8(m-1)].
\end{eqnarray*}
It follows that $nA=[0,6n]\cup [m,2n(m-1)]$ for all $n \ge 4$. In particular, if $6n\ge m$ then $nA=[0,2n(m-1)]$. Therefore,
$$
\beta_G(n)=\beta_A(n)=\left\{
\begin{array}{ccl}
0 & \textrm{if} & n=2,\\
1 & \textrm{if} & n=3,\\
4 & \textrm{if} & 4 \le n \le k+3,\\
6k-6n+22 & \textrm{if} & n \ge k+4.
\end{array}
\right.
$$
Hence $\buchset(S)=[3,k+3]$, as claimed.
\end{proof}

\begin{proposition}
For $k\geq 1$ and $i\in\{1,2,3\}$, let $S_i$ be the numerical semigroup with $G_i=\N\setminus S_i$ given by
$$
\begin{array}{rcl}\label{eq gapset G 3}
G_1 & = & [1,m_1-1] \sqcup \{2m_1-6, 2m_1-2, 2m_1-1\},\\
G_2 & = & [1,m_2-1] \sqcup \{2m_2-10,2m_2-4, 2m_2-3, 2m_2-2\},\\
G_3 & = & [1,m_3-1] \sqcup \{2m_3-10, 2m_3-9, 2m_3-2\},
\end{array}
$$
where $m_1=4k+22$, $m_2=7k+44$ and $m_3=5k+55$, respectively. Then
$$
\buchset(S_1)  = [4,k+4], \;\;
\buchset(S_2) = [5,k+5], \;\;
\buchset(S_3) = [6,k+6].
$$
\end{proposition}
\begin{proof} Similar to the proofs of Propositions \ref{prop [2,k+2]} and \ref{eq gapset G 3}. We omit it here.
\end{proof}

Having realized all finite integer intervals $I$ with $|I| \ge 2$ and $\min(I) \in [2,6]$ as $I=\buchset(S)$ for a suitable numerical semigroups $S$, is it possible to do the same for all finite integer intervals $I$ with $\min(I) \ge 7$? We do not know in general. But here is a particular case where $\min(I)$ can be arbitrarily large. It is based on a family of numerical semigroups found in~\cite{K}.

\begin{proposition} For any integer $k \ge 1$, there is a numerical semigroup $S$ such that $\buchset(S) = [7+2k,7+4k]$.
\end{proposition}

\begin{proof} For $k \ge 1$, let $S$ be the numerical semigroup minimally generated by the set $T_1 \cup T_2 \cup T_3$, where
\begin{eqnarray*}
T_1 & = & [44 + 27 k + 4 k^2,79 + 51 k + 8 k^2], \\
T_2 & = & [81 + 51 k + 8 k^2, 84 + 53 k + 8 k^2], \\
T_3 & = & [87 + 53 k + 8 k^2,87 + 54 k + 8 k^2].
\end{eqnarray*}
The corresponding gapset $G=\N \setminus S$ is then given by
$$
G = [1,43 + 27 k + 4 k^2] \cup \{80 + 51 k + 8 k^2, 85 + 53 k + 8 k^2, 86 + 53 k + 8 k^2 \}.
$$
Let $A=(86 + 53 k + 8 k^2)-G$. Then $$A=[0,1]\sqcup \{6+2k\}\sqcup [43 + 26 k + 4 k^2,85 + 53 k + 8 k^2],$$
of cardinality $|A|=46 + 27 k + 4 k^2$. The $n$-fold sumsets of $A$ are then given by
\begin{multline}\label{A_para_komeda}
nA=[0,n] \cup \left(\bigcup_{i=1}^{n}\left[ i (6 + 2 k), i (6 + 2 k) + n - i\right]\right)\\
\cup \left[43 + 26 k + 4 k^2,(85 + 53 k + 8 k^2) n\right].
\end{multline}

\noindent
$\bullet$ Assume first $2\leq n<6+2k$. In this case, we have
\begin{equation*}
\begin{split}
& 0<n<6+2k<6+2k+n-1<\dots < (n-1)(6+2k) \\
& < (n-1)(6+2k)+n-1 <n(6+2k)<(7+2k)(6+2k)+1 \\
& = 43 + 26 k + 4 k^2< (85 + 53 k + 8 k^2) n.
\end{split}
\end{equation*}
Thus, all the sets appearing in (\ref{A_para_komeda}) are disjoint and the cardinality of $nA$ is equal to
\begin{multline*}(n+1)+\sum _{i=1}^n i+\left(\left(8 k^2+53 k+85\right) n-(4 k^2+26 k+43)+1\right)=\\-41 - 26 k - 4 k^2 + (173 n)/2 + 53 k n + 8 k^2 n + n^2/2.
\end{multline*}
Thus,
\begin{multline}\label{caso1}
\beta_G(n)=(-41 - 26 k - 4 k^2 + (173 n)/2 + 53 k n + 8 k^2 n + n^2/2)\\-\left(4 k^2+27 k+46-1\right) (2 n-1)=\\
(4+k)-(\frac 7 2 +k)n+\frac 1 2 n^2
\end{multline}
for every $n\in [2,5+2k]$.

The only difference between the case $n=6+2k$ and the previous one is that the sets $[0,n]$ and $[6+2k,6+2k+n-1]$ have a nonempty intersection, equal to $\{6+2k\}$.
Replacing $n$ by $6+2k$ and subtracting one, we obtain $\beta_G(6+2k)=0$.

\smallskip
\noindent
$\bullet$ Assume now $6+2k < n \leq 11+4k$.
The sequence of sets $[0,n]$ and
$$
[6+2k,6+2k+n-1], \dots, [(n-5-2k)(6+2k),(n-5-2k)(6+2k)+(5+2k)]
$$
verifies that the intersection of any two consecutive terms is nonempty. Moreover, their union is the interval $[0,(n-5-2k)(6+2k)+(5+2k)]$ whose cardinality is equal to $(n-5-2k)(6+2k)+(5+2k)+1$.
For $i=n-4-2k,\dots,6+2k$ the intervals are disjoint with all the others sets appearing in the expression (\ref{A_para_komeda}); the cardinality of the union of these sets is equal to $\displaystyle \sum _{i=n-2 k-5}^{2k+5} i$.
For every $i=7+2k,\dots,n$ the intersection
$$[ i (6 + 2 k), i (6 + 2 k) + n - i] \cap [43+26k+4k^2,(8 k^2+53 k+85)n]$$ is nonempty, except for $n=7+2k$.
Since $(7+2k)(6+2k)=42+26k+4k^2$,
the set  
$$\left(\bigcup_{i=7+2k}^{n}\left[ i (6 + 2 k), i (6 + 2 k) + n - i\right]\right)\\
\cup \left[43 + 26 k + 4 k^2,(85 + 53 k + 8 k^2) n\right]$$ is equal to $[42+26k+4k^2,(85 + 53 k + 8 k^2) n]$, and the cardinality of this set is $\left(\left(8 k^2+53 k+85\right) n-4 k^2-26 k-42\right)+1$.
Putting all the above together, we have that if $6+2k < n \leq 11+4k$, the set $nA$ has cardinality equal to
\begin{multline*}
    ((n-5-2k)(6+2k)+(5+2k)+1)+\\
    \sum _{i=n-2 k-5}^{2k+5} i+\left(\left(8 k^2+53 k+85\right) n-4 k^2-26 k-42\right)+1=\\
    -65 - 46 k - 8 k^2 + (193 n)/2 + 57 k n + 8 k^2 n - n^2/2,
\end{multline*}
and therefore
\begin{multline}\label{caso3}
\beta_G(n)=
-65 - 46 k - 8 k^2 + (193 n)/2 + 57 k n + 8 k^2 n - n^2/2\\-(2 n - 1) (46 + 27 k + 4 k^2 - 1)=\\
(-20-19 k-4 k^2)+(3 k +\frac{13 }{2})n-\frac{n^2}{2}
\end{multline}
for every $n\in [6+2k,11+4k]$.

\smallskip
\noindent
$\bullet$ Finally, assume $11+4k<n$. The set
$$
[0,n]\cup
(\bigcup_{i=1}^{6+2k}[i(6+2k),i(6+2k)+n-i])
\cup [43 + 26 k + 4 k^2,(85 + 53 k + 8 k^2) n]$$ is equal to $[0,(85 + 53 k + 8 k^2) n]$ and the remaining intervals are contained in this union.
So we have $nA=[0,(85 + 53 k + 8 k^2) n]$ and therefore
\begin{multline}\label{caso4}
\beta_G(n)=-\left(4 k^2+27 k+46-1\right) (2 n-1)+\left(8 k^2+53 k+85\right) n+1=\\(4k^2+27k+46)-(k+5)n
\end{multline}
for every $n>11+4k$.

Combining $\beta_G(6+2k)=0$ with the formulation of (\ref{caso1}), (\ref{caso3}) and (\ref{caso4}) for $\beta_G(n)$, we get the following formulas:
$$
\beta_G(n)=\left\{
\begin{array}{lcl}
(4+k)-(\frac 7 2 +k)n+\frac 1 2 n^2 & \textrm{if} & 2\leq n< 6+2k,\\
0 & \textrm{if} & n=6+2k,\\
(-20-19 k-4 k^2)+(3 k +\frac{13 }{2})n-\frac{n^2}{2} & \textrm{if} & 6+2 k<n\leq 11+4 k,\\
(4k^2+27k+46)-(k+5)n & \textrm{if} & 11+4 k<n.
\end{array}
\right.
$$

Let $k \ge 1$ be fixed. For $2\leq n<6+2k$, the formula of $\beta_G(n)$ is a degree two polynomial in $n$ with positive leading coefficient such that $\beta_G(2)=-1-k<0$ and $\beta_G(5+2k)=-1-k<0$. We have therefore $\beta_G(n)<0$ for every $n=2,\dots,5+2k$.

If $n>11+4k$, we now have that $\beta_G(n)$ is a degree one polynomial with negative leading coefficient and such that $\beta_G(12+4k)=-14-5k<0$. So $\beta_G(n)<0$ for every $n>11+4k$.

Finally, if $6+2 k<n \leq  11+4 k$ the function $\beta_G(n)$ is a degree two polynomial in $n$ with negative leading coefficient. As in addition $\beta_G(7+2k)=1+k$, $\beta_G(7+4k)=1$ and $\beta_G(8+4k)=-k$, the only positive values that we have in this part are for $n\in [7+2k,7+4k]$.

Since $\beta_G(n) \le 0$ except for $n\in [7+2k,7+4k]$, the set $\buchset(S)$ is equal to $[7+2k,7+4k]$.
\end{proof}

\section{Concluding remarks}\label{sec conclusion}
The current knowledge on the structure of $\buch(A)$ for finite subsets $A \subset \Z$ is very scarce, even for gapsets. Do they have some special shape or property? We end this paper with three questions based on the few currently available observations.

\begin{question} Let $A \subset \Z$ be a finite subset, or more specifically a gapset. Is the set $\buch(A)$ always an interval of integers?
\end{question}

\begin{question} Even more so, is the function $\beta_A(n)$ unimodal?
\end{question}

\begin{question} In sharp contrast with the above questions, let $T \subset 2+\N$ be any finite subset. Does there exist a finite subset $A \subset \N$, or more specifically a gapset, such that
$\buch(A)=T$?
\end{question}

{\bf Acknowledgement}. Part of this paper was written during a visit of the first-named author to the Universidad de C\'{a}diz (Spain) which was partially supported by {\em Ayudas para Estancias Cortas de Investigadores} (EST2019-039, Programa de Fomento e Impulso de la Investigaci\'{o}n y la Transferencia en la Universidad de C\'{a}diz).  The second, third and fourth-named authors were partially supported by Junta de Andaluc\'{\i}a research group FQM-366 and by the project MTM2017-84890-P (MINECO/FEDER, UE).


\medskip

\noindent
{\small
\textbf{Authors' addresses:}

\medskip
\noindent
Shalom Eliahou\textsuperscript{a,b}

\noindent
\textsuperscript{a}Univ. Littoral C\^ote d'Opale, UR 2597 - LMPA - Laboratoire de Math\'ematiques Pures et Appliqu\'ees Joseph Liouville, F-62100 Calais, France\\
\textsuperscript{b}CNRS, FR2037, France\\
\textbf{e-mail:} eliahou@univ-littoral.fr

\medskip
\noindent
Juan Ignacio Garc\'ia-Garc\'ia\textsuperscript{c,d}\\
\textbf{e-mail:} ignacio.garcia@uca.es

\medskip
\noindent
Daniel Mar\'in-Arag\'on\textsuperscript{c}\\
\textsuperscript{c}Departamento de Matem\'aticas, Universidad de C\'adiz, E-11510 Puerto Real (C\'{a}diz, Spain).\\
\textbf{e-mail:} daniel.marin@uca.es

\medskip
\noindent
Alberto Vigneron-Tenorio\textsuperscript{d,e}\\
\textsuperscript{e}Departamento de Matem\'aticas, Universidad de C\'adiz, E-11406 Jerez de la Frontera (C\'{a}diz, Spain).\\
\textsuperscript{d}INDESS (Instituto Universitario para el Desarrollo Social Sostenible), Universidad de C\'adiz, E-11406 Jerez de la Frontera (C\'{a}diz, Spain).\\
\textbf{e-mail:} alberto.vigneron@uca.es

}

\end{document}